%% file: winding-road.tex
\documentclass[12pt]{article}
\usepackage{latexsym}
\usepackage{amsmath,amssymb}
\usepackage{xcolor,graphicx}

	\newcommand{\abs}[1]{\left|	#1	\right|}
	\renewcommand{\epsilon}{\varepsilon}

	\newcommand{\set}[2]{ \left\{  #1  \colon  #2  \right\} }
	\newcommand{\oBallin}[3]{\mathrm{B}_{#1}^{#2}\left(#3 \right)}

	\newcommand{\bigset}[2]{ \big\{  #1  \colon  #2  \big\} }
	\newcommand{\smset}[1]{ \left\{  #1    \right\} }
	
	\newcommand{\scp}[2]{\left\langle #1 , #2 \right\rangle}
	\newcommand{\norm}[1]{\ensuremath{\left\| #1 \right\|}}

	\newcommand{\seqkZ}[1] {\left(	#1	\right)_{k\in\Z}}
\newcommand{\N}{\mathbb{N}}
\newcommand{\Z}{\mathbb{Z}}

\newcommand{\R}{\mathbb{R}}

\newcommand{\func}[5]{#1 \colon #2 \longrightarrow #3 \colon #4 \mapsto #5}
\newcommand{\smfunc}[3]{#1 \colon #2 \longrightarrow #3}

\newcommand{\caseelsepunkt}[3]{\left\{ \begin{array}{r l} #1 & \ 
\mathrm{if} \  #2 \\ #3 & \ \mathrm{else.} \end{array} \right.}
\newcommand{\Hilb}{\mathcal H}
\newcommand{\II}{]{-1},1[}
\newtheorem{thm}{Theorem}[section]
\newtheorem{la}[thm]{Lemma}
\newtheorem{Defn}[thm]{Definition}
\newtheorem{Remark}[thm]{Remark}

\newenvironment{defn}{\begin{Defn}\rm}{\end{Defn}}

\newenvironment{proof}{{\noindent\bf Proof.}}%
                  {\nopagebreak\hspace*{\fill}$\Box$\medskip\medskip\par}   
\newcommand{\Punkt}{\nopagebreak\hspace*{\fill}$\Box$}
\newcommand{\ve}{\varepsilon}
\newcommand{\at}{\symbol{'100}}
\newcommand{\sub}{\subseteq}
\DeclareMathOperator{\Supp}{supp}
\DeclareMathOperator{\rB}{B}
\newcommand{\dd}{d_\gamma}
\newcommand{\NB}{\mathcal N}
\newcommand{\vnull}{\frac{\eta'(0)}{\norm{\eta'(0)}}}
\begin{document}
\begin{center}
{\Large\bf Bounded solutions of finite lifetime\vspace{2mm}
to differential equations in Banach spaces}\\[7mm]
{\bf Rafael Dahmen and Helge Gl\"{o}ckner}\vspace{4mm}
\end{center}
\begin{abstract}\vspace{.3mm}\noindent
Consider a
smooth vector field
$f\colon \R^n\to\R^n$ and a maximal
solution $\gamma\colon \,]a,b[\,\to \R^n$
to the ordinary differential equation
$x'=f(x)$.
It is a well-known fact that,
if $\gamma$ is bounded, then
$\gamma$ is a global solution, i.e.,
$\,]a,b[\,=\R$.
We show by example that this conclusion
becomes invalid if $\R^n$ is replaced
with an infinite-dimensional Banach space.
\end{abstract}
{\footnotesize {\em Classification}:
Primary 34C11;
secondary
26E20,
34A12,
34G20,
37C10,
34--01\\[2mm]
%
%
%
%
}\\[8mm]
\noindent
{\large\bf Introduction and statement of result}\\[2.5mm]
The starting point for our journey
is a well-known result in the theory of ordinary differential
equations:
If the image $\gamma(]a,b[)$ of a
maximal solution $\gamma\colon\,]a,b[\,\to U$
to a differential equation
\[
x'=f(x)
\]
with locally Lipschitz right-hand side
$f\colon U\to \R^n$
on an open subset $U\sub \R^n$
is relatively compact in~$U$,
then~$\gamma$ is globally defined,
i.e., $]a,b[\,=\R$
(cf.\ \cite[Chapter~I, Theorem 2.1]{Hal},
\cite[Korollar in 4.2.III]{Koe},
\cite[Corollary 2 in \S2.4]{Per}).
In the special
case $U=\R^n$,
this entails that bounded maximal solutions
$\gamma\colon \,]a,b[\,\to\R^n$
are always globally defined,
exploiting that
bounded sets and relatively compact
subsets in~$\R^n$ coincide
by the Theorem of Bolzano-Weierstrass
(see, e.g., Corollaire~1 in \cite[Chapter IV, \S1, no.\,5]{Bou},
or \cite[Lemma~2.4]{Rob} for this fact).\\[3mm]
The first criterion applies equally well
if $\R^n$ is replaced with a Banach space
(cf.\ \cite[Chapter IV, Corollary~1.8]{Lan}).
However,
bounded maximal solutions to ordinary differential
equations in infinite-dimensional
Banach spaces need not be globally defined.
Non-autonomous examples with locally Lipschitz
right hand sides were given in
\cite{Di0} (in the Banach space $c_0$)
and for Banach spaces admitting a Schauder basis
in \cite{De1}, \cite{De2}.
By now, it is known that
the pathology occurs
for suitable autonomous systems
on \emph{every} infinite-dimensional Banach space,
with locally Lipschitz right-hand side~\cite{Kom}.\\[3mm]
In the current note,
we describe an easy, instructive
example of a non-global, bounded solution
to a vector field on a separable Hilbert
space. In contrast to all of the cited literature,
the vector field we construct is not only
locally Lipschitz, but smooth (i.e., $C^\infty$).\\[4mm]
{\bf Theorem.} \emph{There exists a smooth vector field
$f\colon \Hilb\to \Hilb$ on the real Hilbert space
$\Hilb:=\ell^2(\Z)$ of square summable real
sequences $(a_n)_{n\in \Z}$,
such that the ordinary differential equation
$x'=f(x)$ has a bounded maximal solution~$\gamma$
which is not globally defined.}\\[5mm]
Our strategy is to describe, in a first step,
a smooth curve $\gamma\colon \,]{-1},1[\,\to \Hilb$
whose restrictions to $\,]{-1},0]$
and $[0,1[$ have infinite arc length
(Section~\ref{road}).
In a second step,
we construct a smooth vector field
$f\colon \Hilb\to \Hilb$ such that
\[
\gamma'(t)=f(\gamma(t))\quad\mbox{for all~$\,t\in \,]{-1},1[$,}
\]
ensuring that
$\gamma$ is a solution to $x'=f(x)$
(Section~\ref{sec_ConstructionOfVectorField}).
Thus $\gamma$ is not globally defined
and it has to be a maximal solution
because otherwise the arc length
on one of the subintervals would be finite
(contradiction).
\section{The long and winding road}\label{road}
%
%
We fix a function $\smfunc{h}{\R}{\R}$ with the following properties:
\begin{itemize}
 \item [(i)] $h$ is smooth ($C^\infty$)
with compact support inside $[-2,1]$;
 \item [(ii)] $h(-2)=h(-1)=h(1)=0$;
 \item [(iii)] $h(0)=1$;
 \item [(iv)] $h'(t)> 0$ for all $t\in [-1,0[$, $h'(t)<0$
for $t\in \,]0,1[$ (whence $h(t)>0$ there)
and $h'(0)=0$.
\end{itemize}
\begin{figure}[!hbtp]
\centering
\includegraphics[width=10cm]{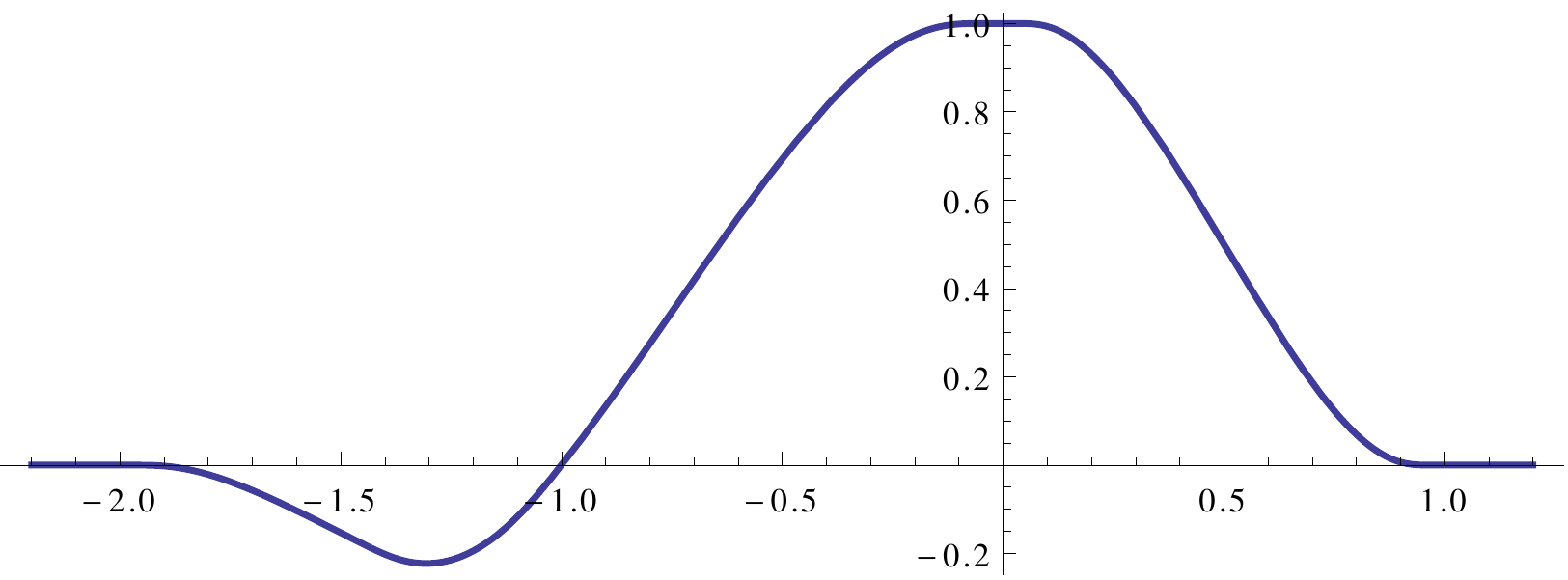} 
\caption{Graph of the function $h$}
\end{figure}


The existence of such a function is shown in
Section~\ref{sec_Functionh}.\\[2mm]
Using this function, we can define a smooth curve with
values in the Hilbert space $\Hilb=\ell^2(\Z)$ via\vspace{-2mm}
\[
\eta\colon \R\to \Hilb,\quad t\mapsto\sum_{k\in\Z}h(t-k)e_k,\vspace{-2mm}
\]
where $\seqkZ{e_k}$ denotes the standard orthonormal basis of~$\Hilb$.
Note that this sum is locally finite since $h$ has compact support;
hence~$\eta$ is smooth.\\[2mm]

 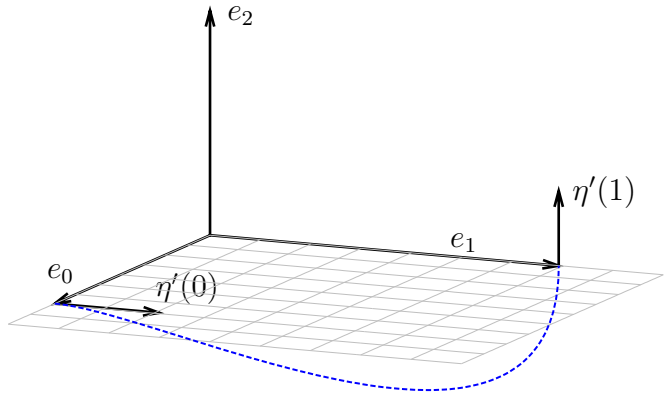
\begin{figure}[h]
\input{3Dcurve_tex}
\caption{The curve $\eta$ on the interval $[0,1]$}
 \end{figure}

\noindent
Calculating the derivative $\eta'(t)=\sum_{k\in\Z}h'(t-k) e_k$,
we see that $\eta'(t)$ is always non-zero.
In fact, if $n\in\Z$ with $t\in [n,n+1[$, then $t-n-1\in[{-1},0[$ and thus
$\langle e_{n+1},\eta'(t)\rangle=h'(t-n-1)\not=0$.\\[2mm]
By construction of~$\eta$, we have $\eta(n)=e_n$ for each $n\in\Z$,
which implies that~$\eta$ has infinite arc length. Since the
real-valued function~$h$ is bounded, it follows that the curve~$\eta$
is (norm-) bounded in the Hilbert space $\Hilb$.\\[2mm]
Next, we fix a diffeomorphism $\smfunc{\varphi}{\II\,}{\R}$ between
the open interval $\II$ and the real line, e.g.\
$\varphi(t)=\tan(\frac{\pi}{2}t)$ or
$\varphi(t)=\frac{t}{1-t^2}$. We now define
\[
\gamma\colon \II\,\to\Hilb,\quad t\mapsto \eta(\varphi(t)).
\]
This curve is just a reparametrization of~$\eta$ and hence shares
some important properties with~$\eta$, namely it is bounded in~$\Hilb$,
the derivative is always nonzero and it has infinite arc length.
However, one important difference is that~$\gamma$ is not globally
defined, so if we are able to show that~$\gamma$ is a maximal
solution to a (time-independent) differential equation, then
our theorem is established.
\section{The surrounding landscape}\label{sec_ConstructionOfVectorField}
Having constructed the curve $\smfunc{\gamma}{\II\,}{\Hilb}$ in
Section~\ref{road}, we shall now define a smooth
vector field $\smfunc{f}{\Hilb}{\Hilb}$ such that~$\gamma$ is a
solution to the differential equation $x'=f(x)$. Since~$\gamma$
(as well as its restriction to $]{-1},0]$ and its restriction to
$[0,1[$)
has infinite arc length by construction, the solution is maximal,
and our theorem follows.\\[2mm]
Write $\langle x,y\rangle:=\sum_{n\in\Z}x_ny_n$
for $x=(x_n)_{n\in\Z}$, $y=(y_n)_{n\in\Z}$ in~$\Hilb$,
and $\|x\|:=\sqrt{\langle x,x\rangle}$. 
We shall use the following facts about distances (to be proven
in Section~\ref{sec_Distance}):
\begin{itemize}
 \item [(a)] The distance function
\[
\dd\colon \Hilb\to\,[0,\infty[,\quad x\mapsto
\inf\bigset{\norm{\gamma(t)-x}}{t\in\,\II}
\]
from the curve~$\gamma$
is continuous on $\Hilb$. In particular, the set
$U_r:=\set{x\in\Hilb}{\dd(x)<r}$ is open and contains the image of~$\gamma$,
for each $r>0$.
 \item [(b)] There is a number $\rho>0$ such that for all $x \in U_\rho$
there exists a unique $\tau(x)\in\, \II$ such that
$\gamma(\tau(x))$ has minimum distance to~$x$, that is
$\norm{\gamma(\tau(x))-x}=\dd(x)$.
 \item [(c)] The map $\smfunc{\tau}{U_\rho}{\,\II}$ is smooth.\footnote{See \cite{Car}, \cite{Die} and \cite{Lan} for differential calculus
on Banach spaces.}
\end{itemize}
The preceding properties entail that the squared distance function
\[
 \func{\dd^2}{U_\rho}{[0,\infty[}{x}{(\dd(x))^2=\norm{\gamma(\tau(x))-x}^2}
\]
is smooth on the neighborhood~$U_\rho$ of~$\gamma$.
This enables us to define the smooth vector field~$f$,
using a suitable cut-off function~$\theta$:
\[
f\colon \Hilb\to \Hilb,\quad x\mapsto
\left\{\begin{array}{cl}
\theta\left(\dd^2(x)\right)\gamma'(\tau(x))& \mbox{if $\,\dd(x)<\rho$;}\\
0 & \mbox{if $\,\dd(x)>\rho/2$.}
\end{array}
\right.
\]
Here, $\smfunc{\theta}{\R}{\R}$ is a fixed smooth function with
$\theta(0)=1$ which vanishes
outside of $[{-\rho^2/4},\rho^2/4]$. It is easily checked
using the properties (a), (b) and (c)
that the map~$f$ is well defined and smooth.
The curve~$\gamma$ is a solution to the associated differential equation,
since for all $t\in\,\II\,$:
\[
f(\gamma(t)) = \theta\Big(\underbrace{\dd^2(\gamma(t))}_{=0}\Big)
\, \gamma'(\ \underbrace{\tau(\gamma(t))}_{=t}\ ) =\gamma'(t).\vspace{-2mm}
\]
This shows that there is a smooth vector field $f$ on $\Hilb$ such that
a maximal solution of the differential equation is bounded but has only
finite lifetime.
\section{Details for Section \ref{road}}\label{sec_Functionh}
In Section~\ref{road}, we used a function~$h\colon \R\to\R$
with certain properties (i)--(iv).\linebreak
We now prove the existence of~$h$.
%
By the Fundamental Theorem,
\[
 \func{h}{\R}{\R}{x}{\int_{-2}^x \!g(t) \, dt}
\]
with a suitable smooth function $\smfunc{g}{\R}{\R}$. This reduces
the problem of finding~$h$ to the problem of finding
a function~$g$ with the following properties:
\begin{itemize}
 \item [(i)$'$]  	$g$ is smooth with support inside $[{-2},1]$
 and integral $\int_{-2}^1g(t) dt =0$;
 \item [(ii)$'$] 	$\int_{-2}^{-1}g(t)dt=0$;
 \item [(iii)$'$] 	$\int_{-1}^0   g(t)dt=1$;
 \item [(iv)$'$] 	$g(t)> 0$ for all $t\in [{-1},0[\,$, $g(t)<0$
for all $t\in\,]0,1[\,$, and $g(0)=0$.
\end{itemize}
%
It remains to construct such a function~$g$. To this end,
we start with a smooth function $\smfunc{\psi}{\R}{\R}$ which is
positive on $\II$ and zero elsewhere, e.g.
\[ 
 \func{\psi}{\R}{\R}{t}{\caseelsepunkt{e^{-\frac{1}{1-t^2}}}{\abs{t}<1}{0\;\;\;}}
\]
Using dilations and translations,
we can create a function $\psi_{]a,b[}$
from the preceding one,
which is positive on any given interval~$]a,b[\,$:
\newcommand{\ff}[2]{\psi_{]#1,#2[}}
\[
 \func{\ff{a}{b}}{\R}{\R}{t}{\psi\left({-1}+2\, \frac{t-a}{b-a}\right).}
\]
Now, we define the function $g$ as
\newcommand{\co}  {A}
\newcommand{\coo} {B}
\newcommand{\cooo}{C}
\begin{equation}\label{makeg}
 g:= \co \cdot \ff{-2}{-1} + \coo \cdot \ff{-3/2}{0} +
 \cooo \cdot \ff{0}{1}		
\end{equation}
with constants $\co,\coo,\cooo\in \R$ determined as follows:\\[2mm]
Condition (iii)$'$ requires that
$B=(\int_{-1}^0\psi_{]{-3/2},0[}(t)\,dt)^{-1}$. Thus $B>0$.\\[2mm]
Condition (ii)$'$ requires that $A\int_{-2}^{-1}\psi_{]{-2},{-1}[}(t)\,dt
=-B\int_{-3/2}^{-1}\psi_{]{-3/2},0[}(t)\,dt$ with~$B$
as just determined. This equation can
uniquely be solved
for~$A$ (with $A<0$).\\[2mm]
Condition~(i)$'$ requires that
$C\int_0^1\psi_{]0,1[}(t)\,dt=-\int_{-2}^0g(t)\,dt=-1$
(where we used~(iii)). This equation
can be solved uniquely for~$C$
(with $C<0$).\\[2mm]
Also~(iv)$'$ holds as $g(0)=0$ by (\ref{makeg}),
$g(t)=B\,\psi_{]{-3/2},0[}(t)>0$ for $t\in\,[{-1},0[$
and $g(t)=C\psi_{]0,1[}(t)<0$ for $t\in\,]0,1[$.
\section{Details for Section \ref{sec_ConstructionOfVectorField} }\label{sec_Distance}
In this section,
we prove the facts (a), (b) and (c)
which were used in Section~\ref{sec_ConstructionOfVectorField} to construct
the vector field~$f$.

(a) is easy to show: In fact, if a metric space $X$ is given and
$A\subseteq X$ is a non-empty subset, then the distance function
  \[
   \func{d_A}{X}{[0,\infty[}{x}{\inf_{a\in A} d(x,a)}
  \]
is the infimum of a family of Lipschitz continuous functions
on~$X$ with Lipschitz constant~$1$. Hence~$d_A$
is Lipschitz with constant~$1$ as well.\\[2.3mm]
We now prove~(b) and~(c)
using a so-called tubular neighborhood (a standard tool
in differential geometry~\cite{Lan}).
No familiarity
with this method is presumed:
All we need can be achieved
directly, by elementary arguments.\\[2.3mm]
We start with two easy lemmas concerning rotations:
\begin{la}[Rotation in {\boldmath$\R^2$}]\label{lemma_rotation_2D}
Let $v,w\in\R^2$ be vectors in $\R^2$ with norm~$1$ and let
$\smfunc{R_{v,w}}{\R^2}{\R^2}$ be the rotation around the origin
mapping~$v$ to~$w$.
Then $R_{v,w}(w)= 2\scp{v}{w} w -v$.
\end{la}
\begin{proof}
Passing to a different coordinate system if necessary,
we may assume that $v=(1,0)$ and $w=(\cos\alpha,\sin\alpha)$.
Then
\begin{eqnarray*}
R_{v,w}(w)&=&(\cos (2\alpha),\; \sin (2\alpha))
\;=\; (\cos^2\alpha-\sin^2\alpha,\; 2\cos\alpha\sin\alpha)\\
&=& (2\cos^2\alpha-1,\; 2\cos\alpha \sin\alpha)\\
&=&
2\langle e_1,(\cos\alpha,\sin\alpha)\rangle\,(\cos\alpha,\sin\alpha)-(1,0),
\end{eqnarray*}
which indeed coincides with $2\langle v,w\rangle w-v$.\vspace{-3mm}
\end{proof}
\begin{la}[Rotation in a real Hilbert space]\label{lemma_rotation_Hilb}
Let $v,w\in\Hilb$ be vectors of norm~$1$. We assume that $v\neq -w$.
Let $\smfunc{R_{v,w}}{\Hilb}{\Hilb}$ be the rotation around~$0$
taking~$v$ to~$w$ and fixing every vector orthogonal to~$v$ and~$w$.
Then the map $R_{v,w}$ is given by the formula
 \[
R_{v,w}(x) = x 
+ \frac{ (2\scp{v}{w}+1)\scp{x}{v} -\scp{x}{w} }{1+\scp{v}{w}} \,w
- \frac{ \scp{x}{v+w} }{1+\scp{v}{w}} \, v\quad\mbox{for all $\,x\in\Hilb$.}
 \]
 In particular, the result depends smoothly on all parameters.
\end{la}
\begin{proof}
 Since both sides of the equation are linear in~$x$, it is enough to
check the following three special cases
(where we used Lemma~\ref{lemma_rotation_2D} for the second):\,\footnote{As
usual, for a subset $Y\subseteq \Hilb$ we write
$Y^\perp:=\{x\in \Hilb\colon (\forall y\in Y)\;\langle x,y\rangle=0\}$.}
 \[
R_{v,w}(v) = w; \quad\;
R_{v,w}(w)= 2\scp{v}{w}w-v; \quad\;
R_{v,w}(x)=x\, \hbox{ for all }\,x\in \smset{v,w}^\bot.
 \]
All three cases are settled by straightforward calculations.
\end{proof}
Note that the map $R_{v,w}$ is not defined in the case that $v=-w$
as the denominator $1+\scp{v}{w}$ becomes zero.
\begin{defn}
 By the \emph{normal bundle of the curve} $\eta$, we mean the subset
 \[
  \NB :=\set{(\eta(t),v)}{t\in\R, v\in\Hilb \hbox{ with }\scp{\eta'(t)}{v}=0}
 \]
of $\Hilb\times\Hilb$.
 It consists of all vectors with basepoint on the curve which are
perpendicular to the curve.
\end{defn}
Although the set $\NB$ carries the structure of a smooth vector bundle,
we need not use the theory of vector bundles in what follows.
Recall that $\eta'(t)\not=0$ for all $t\in \R$.
It is useful to record further properties of~$\eta$.
We shall use that
\begin{equation}\label{eqnrho0}
\rho_0:=\min_{t\in [0,1]}\sqrt{h(t)^2+h(t-1)^2}>0\vspace{-1mm}
\end{equation}
as $h(t)>0$ for all $t\in\,]{-1},1[$.
\begin{la}\label{newlem}
\begin{itemize}
\item[{\rm(a)}]
$\eta\colon \R\to\Hilb$ is injective.
\item[{\rm(b)}]
If $n\in \Z$ and $t\in [n,n+1[$,
then $\eta(t)$ is a linear combination
of $e_n$, $e_{n+1}$ and $e_{n+2}$.
\item[{\rm(c)}]
$\|\eta(s)-\eta(t)\|\geq \rho_0$ for all $s,t\in\R$
such that $|s-t|>3$.
\item[{\rm(d)}]
$\frac{\eta'(t)}{\|\eta'(t)\|}\not=-\frac{\eta'(0)}{\|\eta'(0)\|}$
for all $t\in\R$.
\end{itemize}
\end{la}
\begin{proof}
(a)
Let $t\leq s$ in~$\R$ such that $\eta(t)=\eta(s)$.
There is $n\in \Z$ such that $t\in [n,n+1[$.
If $s\geq n+1$ was true, then $\langle e_n,\eta(s)\rangle=h(s-n)=0$
(as $\Supp(h)\subseteq [{-2},1]$)
while $\langle e_n,\eta(t)\rangle = h(t-n)>0$ (since $t-n\in [0,1[$).
Thus we would get $\eta(s)\not=\eta(t)$, a contradiction.
As a consequence, $s\in [n,n+1[$
as well. Now $h(s-n)=\langle e_n,\eta(s)\rangle=
\langle e_n,\eta(t)\rangle=h(t-n)$ implies
that $s=t$, using that $h|_{[0,1]}$ is strictly decreasing and hence
injective.

(b) Let $m\in \Z$ with $h(t-m)=\langle e_m,\eta(t)\rangle\not=0$.
Since $\Supp(h)\subseteq [{-2},1]$,
we deduce that $t-m\in \,]{-2},1[$,
whence $m\in \,]t-1,t+2[$ and thus $m\in \, ]n-1,n+3[$,
which entails $m\in \{n,n+1,n+2\}$.
The assertion follows.

(c) As $\Supp(h)\subseteq [{-2},1]$
and $|s-t|>3$, we cannot have
both $h(t-k)\not=0$ and $h(s-k)\not=0$
for any $k\in \Z$.
Hence $\eta(t)$ and $\eta(s)$ are orthogonal vectors
and thus $\|\eta(s)-\eta(t)\|=\sqrt{\|\eta(s)\|^2+\|\eta(t)\|^2}
\geq\|\eta(t)\|
=\sqrt{\sum_{k\in\Z}h(t-k)^2}\geq
\sqrt{h(t-n-1)^2+h(t-n)^2}
\geq\rho_0$, with~$n$ as in~(b).

(d) Note that
$\eta'(0)=h'(-1)e_1$ (as $h'(-2)=h'(0)=h'(1)=0$),
where $h'({-1})>0$.
If~$\eta'(t)$ was a negative real multiple
of~$\eta'(0)$ and hence of~$e_1$, then
$h'(t-1)\!=\!\langle e_1,\eta'(t)\rangle\!<\!0$,
thus $t-\!1\in \, ]{-2},{-1}[$ or
$t-\!1\in \,]0,1[$ (as $h'|_{[{-1},0]}\geq 0$).
In the first case, $\langle e_0,\eta'(t)\rangle=h'(t)>0$,
contrary to $\eta'(t)\in -\R \,e_1$.
In the second case, $\langle e_2,\eta'(t)\rangle=h'(t-2)>0$,
contrary to $\eta'(t)\in -\R \, e_1$.\vspace{-3mm}
\end{proof}
\begin{la}[Global parametrization of the normal bundle of {\boldmath$\eta$}]
\label{lemma_parametrization_of_normal_bundle}$\;$\\
Let $\Hilb_0:=\smset{\eta'(0)}^\bot$ and
$R_t:= R_{\vnull,\frac{\eta'(t)}{\norm{\eta'(t)}}}$ be the rotation
turning $\vnull$ to $\frac{\eta'(t)}{\norm{\eta'(t)}}$, as introduced in
Lemma~{\rm\ref{lemma_rotation_Hilb}}.
Then the following map is a bijection:
\[
\Psi\colon \R\times \Hilb_0\to \NB,
\quad
(t,x)\mapsto \left(\eta(t), R_t (x )\right).
\]
\end{la}
\begin{proof}
 First of all, the map
$\smfunc{R_t=R_{\vnull,\frac{\eta'(t)}{\norm{\eta'(t)}}}}{\Hilb}{\Hilb}$
is defined since~$\eta'(t)$ is never~$0$ and $\frac{\eta'(t)}{\|\eta'(t)\|}
\neq -\frac{\eta'(0)}{\|\eta'(0)\|}$ (by Lemma~\ref{newlem}\,(d)).

 \emph{Injectivity:} Assume $\Psi(t_1,x_1)=\Psi(t_2,x_2)$. Since the
curve $\smfunc{\eta}{\R}{\Hilb}$ is injective (see Lemma~\ref{newlem}\,(a)),
we get $t_1=t_2$. Now,
the rotation map is clearly bijective and hence $x_1=x_2$ which shows
injectivity of~$\Psi$.

 \emph{Surjectivity:} Let $(\eta(t),v)\in \NB$ be given.
Because $R_t$ is a bijective isometry taking
$\eta'(0)$ to a non-zero multiple of~$\eta'(t)$,
we have $R_t(\{\eta'(0)\}^\bot)=\{\eta'(t)\}^\bot$.
Thus $\Psi\Big(\{t\}\times \Hilb_0 \Big)=\{\eta(t)\}\times \{\eta'(t)\}^\bot$,
entailing the surjectivity of~$\Psi$.
\end{proof}
We will use the preceding
parametrization of the normal bundle to construct a
parametrization of a tubular neighborhood of~$\eta$.
Before, we recall
a simple lemma from the theory of metric spaces:
\begin{la}[Local injectivity around a compact set]\label{lemma_local_injectivity}
 \,Let $X$ be a\linebreak
metric space and let $\smfunc{f}{X}{Y}$ be a continuous map
to some topological space~$Y$. We assume that~$f$ is locally injective,
i.e.\ each $x\in X$ has an open neighborhood~$V_x$ in~$X$
on which~$f$ is injective.
Assume furthermore that $f$ is injective when restricted to a
non-empty compact
set $K\subseteq X$. Then~$f$ is injective on an
$\ve$-neighborhood $\rB_\ve(K):=\{x\in X\colon d_K(x)<\ve\}$
of~$K$.
\end{la}
\begin{proof}
 The product space $X\times X$ becomes a metric space if we define
the distance between $(x_1,x_2)$ and $(x_1',x_2')$ as
the maximum of $d(x_1,x_1')$ and $d(x_2,x_2')$.
For $x\in X$ and $(x_1,x_2)\in V_x\times V_x$,
we have $f(x_1)=f(x_2)$ if and only if $x_1=x_2$.
The set~$C$
of all pairs $(x_1,x_2)$ on which~$f$ fails to be injective
can therefore be written as
 \begin{eqnarray*}
  C & := & \set{(x_1,x_2)\in X\times X}{f(x_1)=f(x_2)\hbox
{ and } x_1\neq x_2}\\
&=& \left\{(x_1,x_2)\in (X\times X)\setminus \bigcup_{x\in X}(V_x\times V_x)\colon
f(x_1)=f(x_2)\right\},
 \end{eqnarray*}
showing that~$C$ is a closed subset of
the product $X\times X$. The set~$C$ is disjoint to the
compact set $K\times K$ since~$f$ is injective on~$K$. Let~$\ve$ be
the distance between the sets~$C$ and $K\times K$. It follows that~$f$
is injective on~$\rB_\ve(K)$.\vspace{-3mm}
\end{proof}
\begin{la}[Existence of a tubular neighborhood]\label{lemma_tubular}
Consider the map
\[
\Phi\colon \R\times\Hilb_0 \to \Hilb,\quad (t,x)\mapsto \eta(t) +  R_t(x),
\]
which is the composition of the parametrization map~$\Psi$
from Lemma~{\rm\ref{lemma_parametrization_of_normal_bundle}} and the addition
in the Hilbert space~$\Hilb$.

 Then there exists a constant $\rho>0$ such that~$\Phi$ maps the open set
\[
 \Omega_\rho :=\set{(t,x)\in \R\times \Hilb_0}{\norm{x}<\rho}
\]
diffeomorphically onto the open set
\[
 U_\rho := \set{x\in \Hilb}{d_{\eta(\R)}(x)<\rho}.
\]
Moreover, for all $(t,x)\in\Omega_\rho$, the unique point
on $\eta(\R)$ with minimum distance to $\Phi(x,t)$ is~$\eta(t)$.
\end{la}
\begin{proof}
Observe first that~$\Phi$ is a smooth map as a composition of
smooth maps. Next, we calculate the directional
derivative of $\Phi$ at a point $(t_0,0)$ in a direction $(t,x)$:
\begin{eqnarray*}
\lefteqn{\lim_{s\to 0} \frac{\Phi\left(	(t_0,0) + s(t,x)
   \right)-\Phi(t_0,0)}{s}}\quad \\
&=& \lim_{s\to 0}\frac{1}{s}\left(
   \Phi(t_0+st,sx)-\Phi(t_0,0)	\right)\\
& =& \lim_{s\to 0}\frac{1}{s}\left(
    \eta(t_0+st) + R_{t_0+st}(sx) - \eta(t_0) - R_{t_0}(0)\right)\\
& = & \lim_{s\to 0}\frac{1}{s}\left(
    \eta(t_0+st) - \eta(t_0) \right)+\lim_{s\to 0}R_{t_0+st}(x)\\[1mm]
& = & \eta'(t_0)\cdot t + R_{t_0}(x).
\end{eqnarray*}
 Hence, the derivative of~$\Phi$ at $(t_0,0)$ is
the linear mapping
$\R\times\Hilb_0\to\Hilb$, $(t,x)\mapsto \eta'(t_0)t + R_{t_0}(x)$
which is invertible.\\[2.3mm]
By the Inverse Function Theorem, there is an open
neighborhood $\Omega_{t_0}$ of $(t_0,0)$ in $\R\times\Hilb_0$ such
that~$\Phi|_{\Omega_{t_0}}$ is a diffeomorphism onto its open
image $\Phi(\Omega_{t_0})$.\\[2.3mm]
For the moment, let us restrict our attention to the compact set
$[0,4]\times\smset{0}\subseteq \R\times \Hilb_0$ on which~$\Phi$ is
injective (as so is~$\eta$). Then, by Lemma~\ref{lemma_local_injectivity},
there is $\rho>0$ such that~$\Phi$ is injective on
$[0,4]\times \oBallin{\rho}{\Hilb_0}{0}$
(where $\oBallin{\rho}{\Hilb_0}{0}:=\{x\in \Hilb_0\colon \|x\|<\rho\}$).
Since $[0,4]\times \{0\}$ is covered by open
sets on which~$\Phi$ is a diffeomorphism,
after shrinking~$\rho$
we may assume that $\Phi$
takes $]0,4[\times \oBallin{\rho}{\Hilb_0}{0}$
diffeomorphically onto an open set.
We may also assume that $\rho <  \frac{\rho_0}{2}$,
for~$\rho_0$ as in~(\ref{eqnrho0}).
Then $\Omega_\rho:=\R\times \oBallin{\rho}{\Hilb_0}{0}$
has all the required properties:\\[2.3mm]
Exploiting the self-similarity
of~$\eta$, let us show that
$\Phi$ is injective on the set
$[n,n+4]\times \oBallin{\rho}{\Hilb_0}{0}$
for each $n\in\Z$ and that~$\Phi$ restricts to a diffeomorphism
from $]n,n+4[\,\times \oBallin{\rho}{\Hilb_0}{0}$
onto an open subset of~$\Hilb$.
To this end, let
\[
S_n\colon \Hilb\to\Hilb
\]
be the bijective isometry determined by $S_n(e_k)=e_{k+n}$
for all $k\in \Z$. Then $\eta(t+n)=S_n\eta(t)$
for all $t\in \R$ and thus also $\eta'(t+n)=S_n\eta'(t)$.
Hence
\begin{eqnarray*}
\Phi(t+n,x) &=& \eta(t+n)+R_{t+n}(x)\; =\;
S_n\eta(t)+S_nR_tR_t^{-1}S_n^{-1}R_{t+n}(x)\\
&= & S_n\Phi(t,R_t^{-1}S_n^{-1}R_{t+n}(x)).\vspace{-4mm}
\end{eqnarray*}
The map $\;\,\Theta_n\colon \R\times \Hilb_0\to\R\times\Hilb_0$,
$\;\;\Theta_n(t,x):=(t,R_t^{-1}S_n^{-1}R_{t+n}(x))$\\[1mm]
is a bijection and smooth (using that the mapping
\[
\R\times \Hilb_0\to \Hilb,\quad (t,x)\mapsto
R_t^{-1}(x)=R_{\frac{\eta'(t)}{\|\eta'(t)\|},\frac{\eta'(0)}{\|\eta'(0)\|}}(x)
\]
is smooth). Also $\Theta_n^{-1}$ is smooth,
as $\Theta_n^{-1}(t,x)=(t,R_{t+n}^{-1}S_nR_t(x))$.
Note that $R_t^{-1}S_n^{-1}R_{t+n}$ is a bijective isometry
which fixes $\eta'(0)$
and hence takes~$\Hilb_0=\{\eta'(0)\}^\bot$ onto itself.
Hence~$\Theta_n$
is a diffeomorphism that maps
$]0,4[\times \oBallin{\rho}{\Hilb_0}{0}$
(as well as $[0,4]\times \oBallin{\rho}{\Hilb_0}{0}$)
onto itself.
Now $\Phi(t,x)=S_n\Phi(\Theta_n(t-n,x))$
by the above, whence $\Phi$ takes $]n,n+4[\,\times\oBallin{\rho}{\Hilb_0}{0}$
diffeomorphically onto an open set, and is injective on $[n,n+4]\times \oBallin{\rho}{\Hilb_0}{0}$
(as desired).\\[2.3mm]
$\Phi$ is injective on~$\Omega_\rho$:
Let $(s,x),(t,y)\in \R\times \oBallin{\rho}{\Hilb_0}{0}$
with $\Phi(s,x)=\Phi(t,y)$.
If we had $|s-t|>3$, then
$\|\Phi(s,x)-\Phi(t,x)\|=\|\eta(s)-\eta(t)+R_s(x)-R_t(y)\|
\geq \|\eta(s)-\eta(t)\|-\|R_s(x)\|-\|R_t(y)\|\geq\rho_0-2\rho>0$
would follow, contradiction. Thus $|s-t|\leq 3$
and hence $s,t\in [n,n+4]$ for
some $n\in \Z$.
Thus $(s,x)=(t,y)$,
by injectivity of~$\Phi$ on
$[n,n+4]\times\oBallin{\rho}{\Hilb_0}{0}$.\\[2.4mm]
$\Phi(\Omega_\rho)$ is open and $\Phi|_{\Omega_\rho}$ is a diffeomorphism
onto its image: We just verified that $\Phi|_{\Omega_\rho}$
is injective. Since $\Omega_\rho=\bigcup_{n\in\Z}
]n,n+4[\times
\oBallin{\rho}{\Hilb_0}{0}$
and $\Phi$ takes each of the sets
$]n,n+4[\times \oBallin{\rho}{\Hilb_0}{0}$
diffeomorphically onto an open subset of~$\Hilb$,
the assertion follows.\\[2.3mm]
We now show that $\Phi(\Omega_\rho)=U_\rho$.
Since $\|\eta(t)-\Phi(t,x)\|=\|R_t(x)\|<\rho$
if $(t,x)\in \Omega_\rho$,
we have $\Phi(\Omega_\rho)\subseteq U_\rho$.
For the converse inclusion,
let $p\in U_\rho$. To see that $p\in \Phi(\Omega_\rho)$,
we first 
show that the distance $d_{\eta(\R)}(p)$
is attained, i.e., there is $s\in \R$
such that $d_{\eta(\R)}(p)=\|\eta(s)-p\|$.
If this was wrong, we could
choose a sequence $s_k\in \R$
such that $\|\eta(s_k)-p\|\to d_{\eta(\R)}(p)$.
Then $|s_k|\to\infty$ (otherwise, $s_k$ had a bounded subsequence
inside $[{-R},R]$ for some $R>0$,
and then the minimum of the continuous function
$s\mapsto \|\eta(s)-p\|$ on this compact interval
would coincide with $d_{\eta(\R)}(p)$,
a contradiction).
For each $k\in\N$,
there is $n_k\in\Z$ such that $s_k\in [n_k,n_k+1[$.
Then $|n_k|\to\infty$ as well.
After passing to a subsequence,
we may assume that $s_k-n_k\in [0,1]$
converges to some $\Delta\in [0,1]$.
Writing $p_m:=\langle e_m,p\rangle$ for $m\in\Z$,
we have
$\|p\|^2=\sum_{m\in\Z}p_m^2$ and
$(p_m)_{m\in\Z}\in \ell^2(\Z)$.
Lemma~\ref{newlem}\,(b)
now shows that
\begin{eqnarray*}
\|\eta(s_k)-p\|^2 &=&
|h(s_k-n_k)-p_{n_k}|^2
+ |h(s_k-n_k-1)-p_{n_k+1}|^2\\
& & \, + \, |h(s_k-n_k-2)-p_{n_k+2}|^2 \;+\!\!
\sum_{m\not\in\{n_k,n_k+1,n_k+2\}}p_m^2.\vspace{-2mm}
\end{eqnarray*}
Letting $k\to\infty$
(and using that $p_m\to 0$ as $|m|\to\infty$),
we deduce that
\[
d_{\eta(\R)}(p)^2\,=\,
h(\Delta)^2
+
h(\Delta-1)^2
+
h(\Delta-2)^2+ \|p\|^2\geq\rho_0^2
\]
and thus $d_{\eta(\R)}(p)\geq\rho_0$.
But $d_{\eta(\R)}(p)<\rho\leq\rho_0$,
contradiction. Hence, there exists
$s\in \R$ such that $d_{\eta(\R)}(p)=\|\eta(s)-p\|$.\\[2.3mm]
By the preceding, the distance
between the points~$\eta(r)$ and~$p$ (as a
function on~$r$) is minimized for $r=s$.
Since $\frac{d}{dr}\|\eta(r)-p\|^2=2\langle\eta'(r),\eta(r)-p\rangle$,
we deduce that
the derivative~$\eta'(s)$ has to be orthogonal to $\eta(s)-p$.
Thus $y:=R_s^{-1}(p-\eta(s))\in \Hilb_0$
and $p=\Phi(s,y)$. Since $\|y\|=\|p-\eta(s)\|=d_{\eta(\R)}(p)
<\rho$, we have $(s,y)\in \Omega_\rho$
and hence $p=\Phi(s,y)\in \Phi(\Omega_\rho)$.
Thus $U_\rho=\Phi(\Omega_\rho)$.\\[2.3mm]
If also $p=\Phi(t,x)$ for some $(t,x)\in \Omega_\rho$,
then $(t,x)=(s,y)$ by injectivity of~$\Phi$
and thus $s=t$. Hence $\eta(t)$ is the unique point
in~$\eta(\R)$ which minimizes $\|\eta(t)-\Phi(t,x)\|$.
\end{proof}
We are now in the position to
prove the facts~(b) and (c)
stated in Section~\ref{sec_ConstructionOfVectorField}.

To prove~(b), we use the number $\rho>0$ constructed in
Lemma~\ref{lemma_tubular} for the curve~$\eta$. Since the curves~$\gamma$
and $\eta$ differ only by a re-parametrization, the existence of a unique
nearest point remains true.

\noindent
To obtain~(c), we may write the function
\[
      \smfunc{\tau}{U_\rho}{\,\II}
\]
which assigns to each point $x\in U_\rho$ the index
$t\in\,\II$ such that~$\gamma(t)$ has minimum distance to~$x$ as follows:
 \[
  \tau=\varphi^{-1}\circ \pi_\R\circ \Phi^{-1}
 \]
where $\smfunc{\Phi}{\Omega_\rho}{U_\rho}$ is the diffeomorphism from
Lemma~\ref{lemma_tubular}, the mapping\linebreak
$\pi_\R\colon \Omega_\rho\to \R$, $(t,x)\mapsto t$
denotes the projection onto the first component and
$\smfunc{\varphi}{\II\,}{\R}$ is the diffeomorphism used to define
the curve $\gamma$.
As a composition of smooth maps, the map~$\tau$ is smooth.
The proof is complete.\,\vspace{-4mm}\Punkt
{\small

{\bf Rafael Dahmen}, Technische Universit\"{a}t Darmstadt,
Schlo\ss{}gartenstr.\ 7,\\
64285 Darmstadt,
Germany; Email: {\tt dahmen@mathematik.tu-darmstadt.de}\\[3mm]
{\bf Helge  Gl\"{o}ckner}, Universit\"at Paderborn, Institut
f\"{u}r Mathematik,\\
Warburger Str.\ 100, 33098 Paderborn, Germany;
Email: {\tt  glockner\at{}math.upb.de}}\vfill
\end{document}

%% file: 3Dcurve_tex.tex
\begingroup
  \makeatletter
  \providecommand\color[2][]{%
    \GenericError{(gnuplot) \space\space\space\@spaces}{%
      Package color not loaded in conjunction with
      terminal option `colourtext'%
    }{See the gnuplot documentation for explanation.%
    }{Either use 'blacktext' in gnuplot or load the package
      color.sty in LaTeX.}%
    \renewcommand\color[2][]{}%
  }%
  \providecommand\includegraphics[2][]{%
    \GenericError{(gnuplot) \space\space\space\@spaces}{%
      Package graphicx or graphics not loaded%
    }{See the gnuplot documentation for explanation.%
    }{The gnuplot epslatex terminal needs graphicx.sty or graphics.sty.}%
    \renewcommand\includegraphics[2][]{}%
  }%
  \providecommand\rotatebox[2]{#2}%
  \@ifundefined{ifGPcolor}{%
    \newif\ifGPcolor
    \GPcolorfalse
  }{}%
  \@ifundefined{ifGPblacktext}{%
    \newif\ifGPblacktext
    \GPblacktexttrue
  }{}%
  \let\gplgaddtomacro\g@addto@macro
  \gdef\gplbacktext{}%
  \gdef\gplfronttext{}%
  \makeatother
  \ifGPblacktext
    \def\colorrgb#1{}%
    \def\colorgray#1{}%
  \else
    \ifGPcolor
      \def\colorrgb#1{\color[rgb]{#1}}%
      \def\colorgray#1{\color[gray]{#1}}%
      \expandafter\def\csname LTw\endcsname{\color{white}}%
      \expandafter\def\csname LTb\endcsname{\color{black}}%
      \expandafter\def\csname LTa\endcsname{\color{black}}%
      \expandafter\def\csname LT0\endcsname{\color[rgb]{1,0,0}}%
      \expandafter\def\csname LT1\endcsname{\color[rgb]{0,1,0}}%
      \expandafter\def\csname LT2\endcsname{\color[rgb]{0,0,1}}%
      \expandafter\def\csname LT3\endcsname{\color[rgb]{1,0,1}}%
      \expandafter\def\csname LT4\endcsname{\color[rgb]{0,1,1}}%
      \expandafter\def\csname LT5\endcsname{\color[rgb]{1,1,0}}%
      \expandafter\def\csname LT6\endcsname{\color[rgb]{0,0,0}}%
      \expandafter\def\csname LT7\endcsname{\color[rgb]{1,0.3,0}}%
      \expandafter\def\csname LT8\endcsname{\color[rgb]{0.5,0.5,0.5}}%
    \else
      \def\colorrgb#1{\color{black}}%
      \def\colorgray#1{\color[gray]{#1}}%
      \expandafter\def\csname LTw\endcsname{\color{white}}%
      \expandafter\def\csname LTb\endcsname{\color{black}}%
      \expandafter\def\csname LTa\endcsname{\color{black}}%
      \expandafter\def\csname LT0\endcsname{\color{black}}%
      \expandafter\def\csname LT1\endcsname{\color{black}}%
      \expandafter\def\csname LT2\endcsname{\color{black}}%
      \expandafter\def\csname LT3\endcsname{\color{black}}%
      \expandafter\def\csname LT4\endcsname{\color{black}}%
      \expandafter\def\csname LT5\endcsname{\color{black}}%
      \expandafter\def\csname LT6\endcsname{\color{black}}%
      \expandafter\def\csname LT7\endcsname{\color{black}}%
      \expandafter\def\csname LT8\endcsname{\color{black}}%
    \fi
  \fi
  \setlength{\unitlength}{0.0500bp}%
  \begin{picture}(7200.00,5040.00)%
    \gplgaddtomacro\gplbacktext{%
      \csname LTb\endcsname%
      \put(1424,2851){\makebox(0,0)[l]{\strut{}$e_0$}}%
      \put(4463,3080){\makebox(0,0)[l]{\strut{}$e_1$}}%
      \put(2783,4797){\makebox(0,0)[l]{\strut{}$e_2$}}%
      \put(2241,2712){\makebox(0,0)[l]{\strut{}$\eta'(0)$}}%
      \put(5385,3440){\makebox(0,0)[l]{\strut{}$\eta'(1)$}}%
    }%
    \gplgaddtomacro\gplfronttext{%
    }%
    \gplbacktext
    \put(0,0){\includegraphics{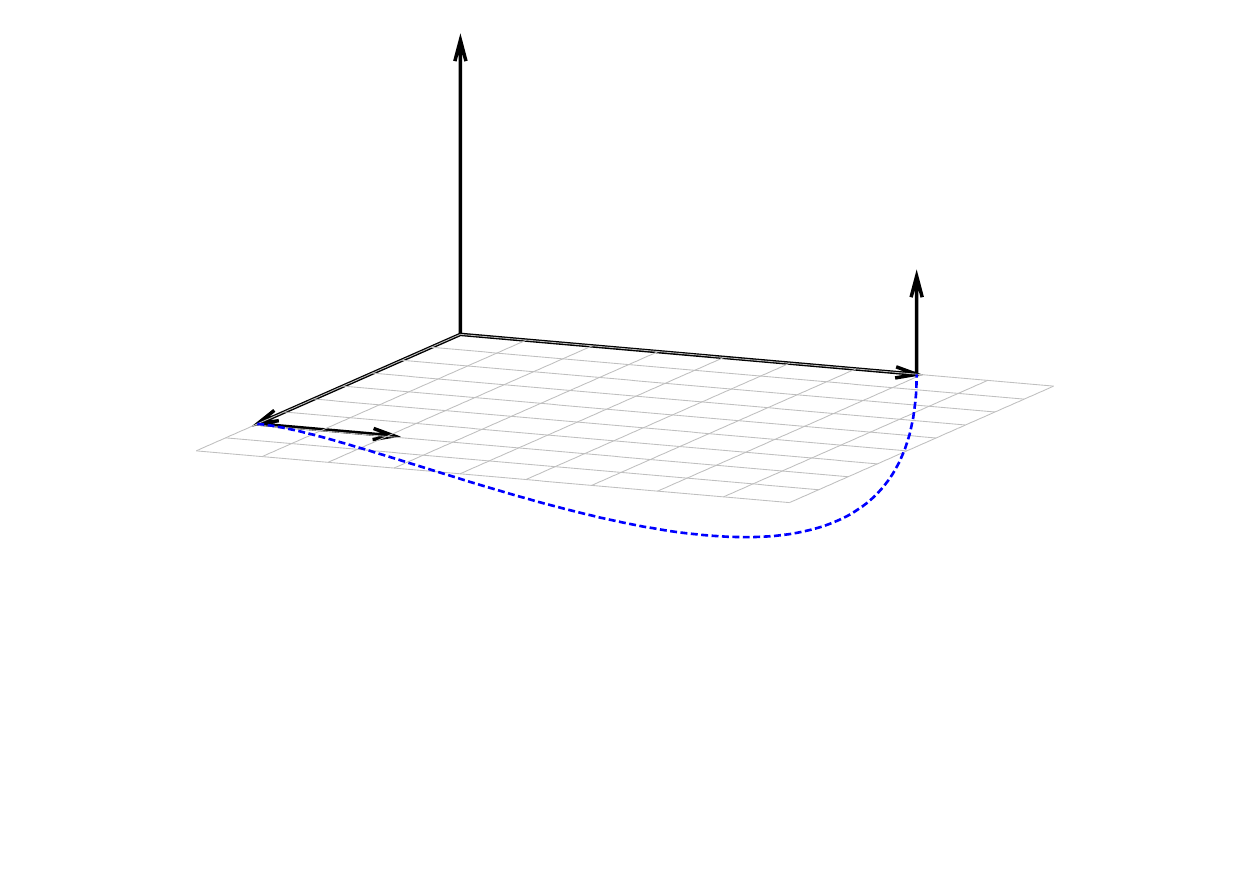}}%
    \gplfronttext
  \end{picture}%
                   \vspace{-3cm}
\endgroup

%% file: winding-road.bbl
\begin{thebibliography}{99}\itemsep+.1pt\vspace{-2mm}
%
%
\bibitem{Bou}
Bourbaki, N.,
``Fonctions d'une variable r\'{e}elle,''
Hermann,
1976.
%
%
\bibitem{Car}
Cartan, H.,
``Differential Calculus,''
Houghton Mifflin Co.,
1971.
%
%
\bibitem{De1}
Deimling, K., ``Ordinary Differential Equations in Banach Spaces,''
Springer,
1977.
%
%
\bibitem{De2}
Deimling, K., ``Multivalued Differential Equations,''
de Gruyter,
1992.
%
%
\bibitem{Di0}
Dieudonn\'{e}, J.,
\emph{Deux exemples d'\'{e}quations diff\'{e}rentielles},
Acta Sci.\ Math.\ (Szeged), {\bf 12B} (1950), 38--40.
%
%
\bibitem{Die}
Dieudonn\'{e}, J.,
``Foundations of Modern Analysis,''
Academic Press,
1969.
%
%
\bibitem{Hal}
Hale, J.\,K., ``Ordinary Differential Equations,''
Wiley-Interscience,
1969.
%
%
\bibitem{Kom}
Komornik, V., P. Martinez, M. Pierre, and J. Vancostenoble,
``\emph{Blow-up}'' \emph{of bounded solutions of differential equations},
Acta Sci.\ Math.\ (Szeged) {\bf 69} (2003), no.\,3-4, 651--657.
%
%
\bibitem{Koe}
K\"{o}nigsberger, K.,
``Analysis 2,'' Springer,
2004. 
%
%
\bibitem{Lan}
Lang, S., ``Differential and Riemannian Manifolds,''
Springer,
1995.
%
%
\bibitem{Per}
Perko, L., ``Differential Equations and Dynamical Systems,''
Springer,
1991.
%
%
\bibitem{Rob}
Robinson, J.\,C., ``Infinite-Dimensional Dynamical Systems,''
Cambridge University Press,
2001.
%
%
\end{thebibliography}
